\documentclass[12pt]{amsart}%
\usepackage{amsfonts}
\usepackage{amsmath}
\usepackage{amssymb}
\usepackage{graphicx}%
\providecommand{\U}[1]{\protect\rule{.1in}{.1in}}
\numberwithin{equation}{section}
\newtheorem{theorem}{Theorem}[section]

\newtheorem{corollary}[theorem]{Corollary}

\newtheorem{example}[theorem]{Example}

\newtheorem{lemma}[theorem]{Lemma}

\newtheorem{proposition}[theorem]{Proposition}
\newtheorem{remark}[theorem]{Remark}

\begin{document}

\title[Function spaces  close to $L^\infty$ with associate space close  to  $L^1$]  {Construction of  function spaces close to  $L^\infty$  with  associate space  close to  $L^1$}

\author{David Edmunds, Amiran Gogatishvili and  Tengiz Kopaliani}

\address{ David Edmunds\\
Department of Mathematics\\
University of Sussex\\ 
Pevensey 2, North-South Road\\
Brighton BN1 9QH, United Kingdom}
\email{davideedmunds@aol.com}
\address{Amiran Gogatishvili \\
Institute of Mathematics of the Academy of Sciences of the Czech Republic \\
\'Zitna 25 \\
115 67 Prague 1, Czech Republic}
 \email{gogatish@math.cas.cz}

\address{Tengiz Kopaliani \\
Faculty of Exact and Natural Sciences\\
I. Javakhi\-shvili Tbilisi State University\\
 University St. 2\\
 0143 Tbilisi, Georgia}
\email{tengiz.kopaliani@tsu.ge}


\thanks{The research was in part supported by the Shota Rustaveli National Science Foundation (SRNSF),  grant  no: 217282,  Operators of Fourier analysis in some classical and new function spaces.
 The research of A.Gogatishvili was partially supported by the grant P201/13/14743S of the Grant agency of the Czech Republic and RVO: 67985840.}
\date{}

\begin{abstract}
The paper introduces a variable exponent space $X$ which has in common  with
$L^{\infty}([0,1])$ the property that  the space  $C([0,1])$ of continuous functions on  $[0,1]$ is a closed linear subspace in it.
The associate space of $X$ contains both the Kolmogorov and the Marcinkiewicz
examples of functions in $L^{1}$ with a.e. divergent Fourier series.
\end{abstract}
\maketitle
\section{Introduction}

One of the fundamental facts about Fourier series is that $L^{1}\left(
\mathbb{T}\right)  $ (where as usual $\mathbb{T}$ denotes the one-dimensional
torus) is not especially pleasant in that, unlike $L^{p}\left(  \mathbb{T}%
\right)  $ when $p>1,$\ it contains a function with a Fourier series that is
almost everywhere divergent. This was first shown by Kolmogorov, who with
remarkable ingenuity constructed such a function. In fact his function belongs
to the space $L\log  \log L \left(  \mathbb{T}\right)  $ that is slightly smaller
than $L^{1}\left(  \mathbb{T}\right)  ,$ and its partial Fourier series
\ diverges unboundedly a.e. Some years later Marcinkiewicz gave an example of
a function in $L^{1}\left(  \mathbb{T}\right)  $ with a.e. divergent Fourier
series even though its partial sums were bounded; various other examples have
been given over the years. From this point of view the gulf between
$L^{1}\left(  \mathbb{T}\right)  $ and $\cup_{p>1}L^{p}\left(  \mathbb{T}%
\right)  $ is wide. The situation is different if, instead of the Lebesgue
spaces $L^{p}\left(  \mathbb{T}\right)  $ with $p>1$ we consider the so-called
variable exponent Lebesgue spaces on $\mathbb{T}.$ These have attracted
considerable attention in recent years, principally because of the role they
play in various applications, such as variational problems with integrands
having non-standard growth. To explain briefly what they are, given a
measurable $p:$ $\mathbb{T}\rightarrow\lbrack1,\infty),$ the Lebesgue space
$L^{p\left(  \cdot\right)  }\left(  \mathbb{T}\right)  $ with variable
exponent $p$ is the space of all measurable functions $f$ on $\mathbb{T}$ such
that for some $\lambda>0,$ $I(\lambda,f):=\int\nolimits_{\mathbb{T}}\left(
\left\vert f(x)\right\vert /\lambda\right)  ^{p(x)}dx<\infty;$ it becomes a
Banach space when endowed with the norm $\left\Vert f\right\Vert _{p(\cdot
)}:=\inf\left\{  \lambda>0:I(\lambda,f)\leq1\right\}  ;$ and it coincides with
the classical $L^{p}$ space when $p$ is constant. It turns out that
\[
L^{1}\left(  \mathbb{T}\right)  =\cup L^{p\left(  \cdot\right)  }\left(
\mathbb{T}\right)  ,
\]
where the union is taken over all measurable $p$ such that $p(x)>1$ a.e. Thus
any function with Fourier series that is divergent a.e. must belong to some
variable exponent space $L^{p\left(  \cdot\right)  }\left(  \mathbb{T}\right),$ 
and just as it is interesting to know that the Kolmogorov function belongs
to $L\log \log  L\left(  \mathbb{T}\right),$ so it is natural to find out to which
space $L^{p\left(  \cdot\right)  }\left(  \mathbb{T}\right) $ it belongs.

In this paper we show that there is a variable exponent space $L^{p\left(
\cdot\right)  }\left(  \mathbb{T}\right)  $, with $1<p(x)<\infty$ a.e., which has in common  with
$L^{\infty}\left(  \mathbb{T}\right)$ the property that  the space  $C\left(\mathbb{T}\right)$ 
of continuous functions on   $\mathbb{T}$ is a closed linear subspace in it.
Moreover, both the Kolmogorov and the Marcinkiewicz functions
belong to $L^{q\left(  \cdot\right)  }\left(  \mathbb{T}\right)  ,$ where
$1/q(x)=1-1/p(x)$ for all $x.$ As might be expected, some knowledge of the
process of construction of the exceptional functions is necessary, and we give
the crucial steps for the convenience of the reader.

\section{Preliminaries}

Let $\Omega$ be a non-empty open subset of $\mathbb{R}^{n};$ let
$\mathcal{M}(\Omega)$ be the set of all measurable and almost everywhere
finite real-valued functions on $\Omega;$ and given any measurable subset $E$
of $\Omega,$ denote by $\left\vert E\right\vert $ the Lebesgue $n-$measure of
$E$ and by $\chi_{E}$ its characteristic function. The open ball in
$\mathbb{R}^{n}$ with centre $x$ and radius $r$ will be denoted by $B(x,r).$
As usual, we say that a linear space $X=X(\Omega)\subset\mathcal{M}(\Omega),$
equipped with a norm $\left\Vert \cdot\right\Vert _{X},$ is a \textit{Banach
function space} (BFS) on $\Omega$ if whenever $f,f_{n},g\in\mathcal{M}%
(\Omega)$ $\left(  n\in\mathbb{N}\right)  ,$ the following axioms hold:

(P1) $\ 0\leq g\leq f$ a.e. implies that $\left\Vert g\right\Vert _{X}%
\leq\left\Vert f\right\Vert _{X};$

(P2) \ $0\leq f_{n}\uparrow f$ a.e. implies that $\left\Vert f_{n}\right\Vert
_{X}\uparrow\left\Vert f\right\Vert _{X};$

(P3) \ $\left\Vert \chi_{E}\right\Vert _{X}<\infty$ if $E\subset\Omega$ and
$\left\vert E\right\vert <\infty;$

(P4) \ given any $E\subset\Omega$ with $\left\vert E\right\vert <\infty,$
there is a constant $C_{E}>0$ such that for all $f\in X,$%
\[
\int\nolimits_{E}fdx\leq C_{E}\left\Vert f\right\Vert _{X}.
\]
The associate space $X^{\prime}$ of a BFS $X$ is the set of all $g\in
\mathcal{M}(\Omega)$ such that $f.g\in L^{1}(\Omega);$ when endowed with the
norm
\[
\left\Vert g\right\Vert _{X^{\prime}}:=\sup\left\{  \left\Vert f.g\right\Vert
_{L^{1}(\Omega)}:\left\Vert f\right\Vert _{X}\leq1\right\}
\]
it is a BFS on $\Omega.$ Moreover, $X^{\prime}$ is a closed, norm fundamental
subspace of the dual $X^{\ast}$ of $X.$ We refer to \cite{BeS} for basic
properties of Banach function spaces.

\ Let $f$ be a measurable, real-valued function on $\Omega.$ Its
\textit{non-increasing rearrangement} $f^{\ast}$ is defined by%
\[
f^{\ast}(t)=\inf\left\{  \lambda\in(0,\infty):\left\vert \left\{  x\in
\Omega:\left\vert f(x)\right\vert >\lambda\right\}  \right\vert \leq
t\right\}  ,t\in\left[  0,\left\vert \Omega\right\vert \right]  .
\]
A BFS $X$ is said to be \textit{rearrangement-invariant} (r.i.) if

(P5) \ $\left\Vert f\right\Vert _{X}=\left\Vert g\right\Vert _{X}$ whenever
$f^{\ast}=g^{\ast}.$

To every r.i. space $X$ $\left(  \Omega\right)  $ there corresponds a unique
r.i. space $\overline{X}\left(  \left(  0,\left\vert \Omega\right\vert
\right)  \right)  $ such that $\left\Vert f\right\Vert _{X(\Omega)}=\left\Vert
f^{\ast}\right\Vert _{\overline{X}\left(  \left(  0,\left\vert \Omega
\right\vert \right)  \right)  }$ for all $f\in X(\Omega).$ This space, endowed
with the norm
\[
\left\Vert f\right\Vert _{\overline{X}\left(  \left(  0,\left\vert
\Omega\right\vert \right)  \right)  }:=\sup_{\left\Vert g\right\Vert
_{X^{\prime}(\Omega)}\leq1}\int\nolimits_{0}^{\left\vert \Omega\right\vert
}f^{\ast}(t)g^{\ast}(t)dt,
\]
is called the \textit{representation space} of\ $X$ $\left(  \Omega\right)  .$

The \textit{fundamental function} of an r.i. space $X$ $\left(  \Omega\right)
$ is the map $\phi_{X}:[0,\left\vert \Omega\right\vert ]\rightarrow
\lbrack0,\infty)$ defined by
\[
\phi_{X}(t)=\left\Vert \chi_{(0,t)}\right\Vert _{\overline{X}\left(  \left(
0,\left\vert \Omega\right\vert \right)  \right)  }\text{ \ }\left(
t\in(0,\left\vert \Omega\right\vert ]\right)  ,\phi_{X}(0)=0.
\]

We now introduce various interesting subspaces of a BFS $X\left(
\Omega\right)  $. A function $f$ in $X$ is said to have \textit{absolutely
continuous norm }in $X$ if $\left\Vert f\chi_{E_{n}}\right\Vert _{X}%
\rightarrow0$ whenever $\left\{  E_{n}\right\}  $ is a sequence of measurable
subsets of $\Omega$ such that $\chi_{E_{n}}\downarrow0$ a.e. The set of all
such functions is denoted by $X_{a}.$ By $X_{b}$ is meant the closure of the
set of all bounded functions in $X.$ Following Lai and Pick \cite{LP}, a
function $f\in X$ is said to have \textit{continuous norm in }$X$ if for every
$x\in\Omega,$ $\lim_{\varepsilon\rightarrow0}\left\Vert f\chi_{B(x,\varepsilon
)}\right\Vert _{X}=0;$ the set of all these functions is written as $X_{c}.$
The connection between this notion and the compactness of Hardy operators from
a weighted BFS $(X,w)$ to $L^{\infty}$ is explored in \cite{LP}; for a
connection with unconditional bases in BFSs see \cite{Kop1}, \cite{Kop2}. \ In
general, the relation between the subspaces $X_{a},X_{b}$ and $X_{c}$ is
complicated:\ for example (see \cite{LN}), there is a BFS $X$ for which
$\left\{  0\right\}  =X_{a}\varsubsetneq X_{c}=X.$

We now focus on the case in which $\Omega$ is a bounded interval in the real
line, taken to be $\left(  0,1\right)  $ for simplicity, although the
arguments will work for any bounded interval $\left(  a,b\right)  .$ Let
$I=[0,1]$ and let $\mathcal{P}(I)$ be the family of all measurable functions
$p:I\rightarrow\lbrack1,\infty).$ When $p\in\mathcal{P}(I)$ we denote by
$L^{p(\cdot)}\left(  I\right)  $ the set of all measurable functions $f$ on
$I$ such that for some $\lambda>0,$%
\[
\int\nolimits_{0}^{1}\left(  \frac{\left\vert f(x)\right\vert }{\lambda
}\right)  ^{p(x)}dx<\infty.
\]
This set becomes a BFS when equipped with the norm%
\[
\left\Vert f\right\Vert _{p(\cdot)}:=\inf\left\{  \lambda>0:\int
\nolimits_{0}^{1}\left(  \frac{\left\vert f(x)\right\vert }{\lambda}\right)
^{p(x)}dx\leq1\right\}  ;
\]
it is often referred to as a (Lebesgue) space with variable exponent. When $p$
is constant, the space coincides with the standard space $L^{p}\left(
I\right)  .$ Spaces with variable exponent, and Sobolev spaces $W^{k,p(\cdot
)}$ based upon them not only have intrinsic interest but also have
applications to partial differential equations and the calculus of variations.
More details will be found in \cite{CUF1} and \cite{DHHR}. For the particular
BFS $X=L^{p(\cdot)}\left(  I\right)  $ the relation between it and its
subspaces $X_{a},X_{b}$ and $X_{c}$ was investigated in \cite{ELN}: we give
some of the results of that paper next.

\begin{proposition}
\label{Proposition 2.1}Let $p \in\mathcal{P}(I)$ and set $X=L^{p(\cdot
)}\left(  I\right)  .$ Then

(i) $X_{a}=X_{c};$

(ii) $X_{b}=X$ if and only if $p(\cdot)\in L^{\infty}(I);$

(iii) $X_{a}=X_{b}$ if and only if
\[
\int\nolimits_{0}^{1}A^{p^{\ast}(x)}dx<\infty\text{ for all }A>1,
\]
where $p^{\ast}$ is the  non-increasing rearrangement of $p$.
\end{proposition}

Further understanding of these relations is given by the following examples
taken from \cite{ELN}.

\begin{example}
\label{Example 2.2} Let $n=1$ and $\Omega=(0,1/e).$ Then

(i) if $p(x)=x^{\alpha}$ with $\alpha<0,$ then $X_{a}\varsubsetneq X_{b};$

(ii) if $p(x)=(\log x^{-1})^{\alpha},$ then $X_{a}=X_{b}$ if $\alpha\in(0,1],$
and $X_{a}\varsubsetneq X_{b}$ if $\alpha>1.$
\end{example}

\section{A variable exponent Lebesgue space close to $L^{\infty}(I)$}

For the remainder of the paper we shall denote by $m$ the function defined,
for every $x\in(0,1]$ by
\[
m(x)=\chi_{(0,1/200)}(x)\log(1/x);
\]
$\left\{  \delta_{k}\right\}  $ will be a sequence of positive numbers with%
\begin{equation}
\lim_{k\rightarrow\infty}\delta_{k}=0\text{ and }\sum\nolimits_{k=1}^{\infty
}\int\nolimits_{0}^{\delta_{k}}\log(1/x)dx<\infty;\label{Eq 3.1}%
\end{equation}
and $\left\{  r_{k}\right\}  $ will be an enumeration of the rationals in
$I=[0,1],$ (or some dense set in $I$).  Moreover, $\widetilde{p}$ will be the function defined on $I$ by%
\begin{equation}
\widetilde{p}(x)=2+\sum\nolimits_{k=1}^{\infty}m(x-r_{k})\chi_{\left(
r_{k},r_{k}+\delta_{k}\right)  }(x).\label{Eq 3.2}%
\end{equation}
Using the elementary fact that
\[
\int\nolimits_{a}^{a+\delta}-\log xdx=\delta+\log\left\{  \left(  \frac
{a}{a+\delta}\right)  ^{a}(a+\delta)^{-\delta}\right\}  <\delta-\delta
\log\delta,
\]
it follows from the monotone convergence theorem and (\ref{Eq 3.1}) that
$\widetilde{p}(x)$ is finite a.e. on $I.$ Moreover,%
\[
1<\text{ ess inf }\widetilde{p}(x),\text{ ess sup }\widetilde{p}(x)=\infty,
\]
and%
\[
L^{\infty}(I)\subset L^{\widetilde{p}(\cdot)}(I)\subset L^{1}(I).
\]

To  investigate further properties of $L^{\widetilde{p}(\cdot)}(I).$ we prove following theorem, (as we know it is new).
 
\begin{theorem}
\label{theorem 3.1} Let $X$ be  a BFS on $I$.  The space $C(I)$ of continuous functions on $I$ is a
closed linear subspace of  $X$ if and only if  there exists a positive constant $c$ satisfying 
\begin{equation} \label{eq3.3}  \|\chi_{(a,b)}\|_X\ge c\quad \text{whenever} \quad  0\le a<b\le 1. 
\end{equation}
\end{theorem}
\begin{proof}
For sufficiency part  it is enough to show that there is a positive constant $C$ such that for every
$f\in C(I),$%
\begin{equation} \label{3.111}
C\left\Vert f\right\Vert _{C(I)}\leq\left\Vert f\right\Vert _{X}\leq\left\Vert f\right\Vert _{C(I)}.
\end{equation}
The second of these inequalities is clear. For the first, let $f\in C(I).$ There exists $x_{0}\in I$ such that $\left\Vert f\right\Vert
_{C(I)}=\left\vert f\left(  x_{0}\right)  \right\vert ;$ there exists
$\varepsilon>0$ such that $\left\vert f(x_{0})\right\vert \leq2\left\vert
f(x)\right\vert $ if $x\in\left(  x_{0}-\varepsilon,x_0+\varepsilon\right)  \cap
I:=E.$ Thus from (\ref{eq3.3}) we see that%
\[
\left\Vert f\right\Vert _{C(I)}=\left\vert f\left(  x_{0}\right)  \right\vert
\leq \frac1 c
\left\vert f\left(  x_{0}\right)  \right\vert \left\Vert \chi
_{E}\right\Vert _{X}\leq\frac 2c\left\Vert f\chi_{E}\right\Vert
_{X}\leq\frac 2c\left\Vert f\right\Vert _{X}.
\]

Necessity. If $C(I)$ is  a  closed subset of $X$, then by the  closed graph theorem, we have the  estimate \eqref{3.111}.
 Let given any interval $(a,b)\subset I $ be given,  if we   take a continuous function $g$ on $I$ such that $g\le \chi_{(a,b)}$ and $\|g\|_{L^\infty}=1$ we get \eqref{eq3.3}. 
\end{proof}

We now establish further properties of $L^{\widetilde{p}(\cdot)}(I).$

\begin{theorem}\label{theorem 3.2} For any $(a,b)\subset I$, we have 
\begin{equation}
\left\Vert \chi_{(a,b)}\right\Vert _{\widetilde{p}(\cdot)}>1/e. \label{Eq 3.3}%
\end{equation}
\end{theorem}
\begin{proof}
Let $(a,b)\subset I$
and let $k\in\mathbb{N}$ be such that $a<s_{k}<$ $s_{k}+\delta_{k}<b$ for some
rational $s_{k}.$ Then
\begin{align*}
\int\nolimits_{a}^{b}\left(  \frac{1}{1/e}\right)  ^{\widetilde{p}(x)}dx  &
\geq\int\nolimits_{s_{k}}^{s_{k}+\delta_{k}}\left(  \frac{1}{1/e}\right)
^{\widetilde{p}(x)}dx\\
&\geq\int\nolimits_{s_{k}}^{s_{k}+\delta_{k}}\exp\left\{
\log\left(  \frac{1}{x-s_{k}}\right)  \right\}  dx\\
&  =\int\nolimits_{0}^{\delta_{k}}\frac{dx}{x}=\infty,
\end{align*}
and so we have \eqref{Eq 3.3}.
\end{proof}

\begin{corollary}
\label{cor3.1} The space $C(I)$ of continuous functions on $I$ is a
closed linear subspace of \ $L^{\widetilde{p}(\cdot)}(I).$
\end{corollary}
\begin{proof} Using Theorem~\ref{theorem 3.2} from  Theorem~\ref{theorem 3.1} we obtain that there is a positive constant $C$ such that for every
$f\in C(I),$%
\[
C\left\Vert f\right\Vert _{C(I)}\leq\left\Vert f\right\Vert _{\widetilde
{p}(\cdot)}\leq\left\Vert f\right\Vert _{C(I)}.
\]
From these estimates the  proof of corollary follows. 
\end{proof}

\begin{remark}
\label{Remark 3.2}
\end{remark}

With the conjugate exponent $p^{\prime}$ defined by $1/p(x)+1/p^{\prime}(x)=1$
\newline$\left(  x\in I,p\in\mathcal{P}(I)\right)  ,$ it is known that
$L^{p^{\prime}(\cdot)}(I)$ is isomorphic to the dual $\left(  L^{p(\cdot
)}(I)\right)  ^{\ast}$\ of $L^{p(\cdot)}(I)$ if and only if $p(\cdot)\in
L^{\infty}(I);$ when ess sup $p(x)=\infty,$ $L^{p^{\prime}(\cdot)}(I)$ is
isomorphic to a proper closed subspace of $\left(  L^{p(\cdot)}(I)\right)
^{\ast}.$ The space $L^{\widetilde{p}(\cdot)}(I)$ constructed above has some
properties similar to those of $L^{\infty}(I).$ For example, the dual of
$L^{\infty}(I)$ is the set of finitely additive measures that are absolutely
continuous with respect to Lebesgue measure. These functionals are extensions
to $L^{\infty}(I)$ of continuous linear functionals on $C(I).$ Since
$C(I)\ $is a closed linear subspace of $L^{\widetilde{p}(\cdot)}(I),$ it
follows from the Hahn-Banach theorem that any bounded linear functional on
$C(I)$ can be extended to $L^{\widetilde{p}(\cdot)}(I).$

\begin{remark}
\label{Remark 3.3}
\end{remark}

Given any interval $(a,b)\subset I,$ there is a no r.i. space $X((a,b))$
different from $L^{\infty}((a,b))$ such that $L^{\infty}((a,b))\subset
X((a,b))$ $\subset L^{\widetilde{p}(\cdot)}((a,b)).$ For by (\ref{Eq 3.3}),
there exists $C>0$ such that $\phi_{X}(t)\geq C$ for all $t\in(0,b-a),$ and
thus $X((a,b))=L^{\infty}((a,b)).$

It is clear that
 \[
L_{a}^{\widetilde{p}(\cdot)}(I)\neq\{0\}.
\]

For each $n\in\mathbb{N}_{0}$ let $E_{n}=\left\{  x\in I:n\leq
\widetilde{p}(x)<n+1\right\}  $ and let $\left\{  G_{n}\right\}  $ be a sequence of
disjoint sets such that$\ G_{n}\subset E_{n}$ and $\left\vert G_{n}\right\vert
<\exp\left(  -e^{n}\right)  ;$ define a function $g$ by
\[ g(x)=\sum\nolimits_{n=0}^{\infty}n\chi_{G_{n}}(x).\]
This function does not belong to $L^{\infty}(I),$ but since $\int
\nolimits_{0}^{1}\left(  Ag(x)\right)  ^{\widetilde{p}(x)}dx<\infty$ for every
$A>1,$ it is in $L_{a}^{\widetilde{p}(\cdot)}(I).$ 

Therefore 
\[ L^\infty \subsetneqq L_{a}^{\widetilde{p}(\cdot)}(I).
\]

\section{A variable exponent Lebesgue space close to $L^{1}(I)$}

Let $\widetilde{q}(\cdot)$ be the conjugate exponent of the function
$\widetilde{p}(\cdot)$ defined by (\ref{Eq 3.2}),  i.e. 
\[ \frac1{\widetilde{q}(\cdot)}+ \frac 1{\widetilde{p}(\cdot)}=1, \quad x\in I,
\]
with the convention that  $1/\infty=0$.
 Note that $\widetilde{q}(x)>1$ for
a.e. $x\in I,$ and that the essential infimum of $\widetilde{q}(x)$ on
every interval $\left(  a,b\right)  \subset I$ is $1;$ moreover,
$L^{\widetilde{q}(\cdot)}(I)$ can be identified with the associate space
$\left(  L^{p(\cdot)}(I)\right)  ^{\prime}.$

The conjugate of the function $x\longmapsto1+m\left(  x-r_{k}\right)
\chi_{\left(  r_{k},r_{k}+\delta_{k}\right)  }(x)$ on the interval $\left(
r_{k},r_{k}+\delta_{k}\right)  $ is $\widetilde{q}_{k},$ where $\widetilde{q}_{k}(x)=1+1/\log\left(
\frac{1}{x-r_{k}}\right)  :\ \ $ thus $\widetilde{q}_{k}(\cdot)$ satisfies the
estimates 
\[ 1- \frac{1}{\widetilde{q}_{k}(x)}\le \frac{1}{\log \frac{1}{\varepsilon}}, \quad  x \in \left(  r_{k}, r_{k}+\varepsilon \right) \quad \text{and} \quad 0<\varepsilon\le \delta_{k}.
\]
Therefore,
\begin{equation} \label{4.11}\varepsilon ^{\frac{1}{(\widetilde{q}_k)_{+}}}\le e \varepsilon,
\end{equation}
where $ (\widetilde{q}_k)_{+}$  is essential supremum of $\widetilde{q}_k (x)$ on the interval $\left(  r_{k}, r_{k}+\delta_{k}\right).$ 
As 
\[ 0<\|\chi _{(a,b)}\|_{\widetilde{p}(\cdot)}\le 1,\]
 by  Corollary 2.23   from \cite{CUF1}
  \[ |(a,b)|^{1/\left(\widetilde{q}_k\right)_{-}}\le   \|\chi _{(a,b)}\|_{\widetilde{q}_k(\cdot)} \le |(a,b)|^{1/\left(\widetilde{q}_k\right)_{+}}\]
and   using \eqref{4.11}, we have
\[\|\chi_{(r_k, r_k+\varepsilon )}\|{\widetilde{q}_k(\cdot)}\approx  \varepsilon, \quad \text{for every}\quad 0<\varepsilon\le \delta_k.\]

Since $\widetilde{q}(x)\leq q_{k}(x)$ on $\left(  r_{k},r_{k}+\varepsilon
\right)  $ when $0<\varepsilon\leq\delta_{k},$ we thus have
\begin{equation}
\varepsilon=\left\Vert \chi_{\left(  r_{k},r_{k}+\varepsilon\right)
}\right\Vert _{L^{1}}\lesssim \left\Vert \chi_{\left(  r_{k},r_{k}%
+\varepsilon\right)  }\right\Vert _{\widetilde{q}(\cdot)}\lesssim\left\Vert
\chi_{\left(  r_{k},r_{k}+\varepsilon\right)  }\right\Vert _{q_{k}(\cdot)}
\lesssim \varepsilon. \label{Eq 4.1}%
\end{equation}
Let $f$ be a non-negative decreasing step function on $\left(  r_{k}%
,r_{k}+\delta_{k}\right)  :$ this can be written as%
\[
f(x)=\sum\nolimits_{i=1}^{\infty}a_{i}\chi_{\left(  r_{k},x_{i}\right)  }(x)
\]
where $r_{k}+\delta_{k}=x_{1}>x_{2}>\cdots>x_{i}>\dots$ and each $a_{i}\geq0.$ Using
(\ref{Eq 4.1}) we obtain%
\begin{align*}
\left\Vert f\right\Vert _{L^{1}}  &  \lesssim \left\Vert f\right\Vert_{\widetilde{q}(\cdot)}=
 \left\Vert \sum\nolimits_{i=1}^{\infty}a_{i}\chi_{\left(  r_{k},x_{i}\right)  }\right\Vert _{_{\widetilde{q}(\cdot)}}\\
&\lesssim \sum\nolimits_{i=1}^{\infty}a_{i}\left\Vert \chi_{\left(  r_{k},x_{i}\right)
}\right\Vert _{_{\widetilde{q}(\cdot)}}
 \approx \sum\nolimits_{i=1}^{\infty}a_{i}\left(  x_{i}-r_{k}\right)\\
&=\left\Vert \sum\nolimits_{i=1}^{\infty}a_{i}\chi_{\left(  r_{k}, x_{i}\right)  }\right\Vert _{L^{1}}
=\left\Vert f\right\Vert _{L^{1}}.
\end{align*}
It follows that for every non-negative decreasing function $f$ on $\left(
r_{k},r_{k}+\delta_{k}\right),$
\begin{equation}
\left\Vert f\chi_{\left(  r_{k},r_{k}+\delta_{k}\right)
}\right\Vert _{_{\widetilde{q}(\cdot)}}\approx\left\Vert f\chi_{\left(
r_{k},r_{k}+\delta_{k}\right)  }\right\Vert _{L^{1}}. \label{Eq 4.2}%
\end{equation}

\begin{lemma}
\label{Lemma 4.1}Let $f=\sum\nolimits_{k=1}^{\infty}f_{k},$ where each $f_{k}$
is a non-negative, non-increasing function on $\left(  r_{k},r_{k}+\delta
_{k}\right)  $ that is zero outside $\left(  r_{k},r_{k}+\delta_{k}\right)  .$
Then
\[
\left\Vert f\right\Vert _{_{_{\widetilde{q}(\cdot)}}}\approx\left\Vert
f\right\Vert _{L^{1}},
\]
where $\approx$ means that the left-hand side is bounded above and below by
positive, constant multiples of the right-hand side that are independent of
the particular $f.$
\end{lemma}
\begin{proof}
As in the proof of (\ref{Eq 4.2}),%
\begin{align*}
\left\Vert f\right\Vert _{L^{1}}  &  \lesssim  \left\Vert f\right\Vert
_{_{_{\widetilde{q}(\cdot)}}}=\left\Vert \sum_{k=1}^{\infty}%
f_{k}\right\Vert _{_{_{\widetilde{q}(\cdot)}}}\leq \sum_{k=1}%
^{\infty}\left\Vert f_{k}\right\Vert _{_{_{\widetilde{q}(\cdot)}}}\\
& \lesssim  \sum_{k=1}^{\infty}\left\Vert f_{k}\right\Vert _{L^{1}%
}=\left\Vert f\right\Vert _{L^{1}}.
\end{align*}
\end{proof}

The space $L^{\widetilde{q}(\cdot)}(I)$  is  subspace of  $L^{1}(I)$ since
$\widetilde{q}(x)>1$ a.e. on $I;$ although it has many bad properties it is
quite like $L^{1}(I)$ in various respects. A simple illustration of this is
given in the next theorem, in which it is shown that there is a function $f\in
L^{\widetilde{q}(\cdot)}(I)$ such that the Hardy-Littlewood maximal function
$Mf$ \ is not integrable over any interval, no matter how small, contained in
$I.$ [We recall that the Hardy-Littlewood operator $M$ is defined by
\[
Mf\left(  x\right)  =\sup_{Q\backepsilon x}\frac{1}{\left\vert Q\right\vert
}\int\nolimits_{Q}\left\vert f(y)\right\vert dy,
\]
where the supremum is taken over all intervals $Q\subset I$ that contain $x.$]
This stems from the bad oscillatory behaviour of $\widetilde{q}$ \ and the
fact that $\widetilde{q}$ is not continuous or strictly greater than $1$ in
any small interval.

\begin{theorem}
\label{Theorem 4.2}There exists $f\in L^{\widetilde{q}(\cdot)}(I)$ such that
the Hardy-Littlewood maximal function $Mf$ \ is not integrable in any interval
$\left(  a,b\right)  \subset I.$
\end{theorem}

\begin{proof}
Let $f$ be defined by%
\[
f(x)=\frac{d}{dx}\left(  \frac{1}{\log(1/x)}\right)  \text{ \ }\left(
x\in(0,1/e)\right)  .
\]
Then $f$ is non-negative, decreasing and integrable on $\left(  0,1/e\right)
;$ for $x\in(0,1/e)$ we have%
\[
Mf(x)\geq\frac{1}{x}\int\nolimits_{0}^{x}\frac{d}{dt}\left(  \frac{1}%
{\log(1/t)}\right)  dt=\frac{1}{x\log(1/x)}.
\]
The function $x\longmapsto\frac{1}{x\log(1/x)}$ is non-negative and
decreasing, but not integrable on $\left(  0,1/e\right)  .$

Now consider the function $g$ defined on $I$ by%
\[
g(x)=\sum\nolimits_{k=1}^{\infty}a_{k}f(x-r_{k})\chi_{\left(  r_{k}%
,r_{k}+\delta_{k}\right)  }(x),
\]
where each $a_{k}>0$ and $\sum\nolimits_{k=1}^{\infty}a_{k}<\infty.$ Use of
Lemma \ref{Lemma 4.1} shows that $f\in L^{\widetilde{q}(\cdot)}(I),$ but
$\left\Vert \chi_{(a,b)}Mf\right\Vert _{L^{1}}=\infty,$ no matter what
interval $\left(  a,b\right)  \subset I$ we choose.
\end{proof}

\bigskip

Note that conditions on the exponent function $p$ sufficient for the validity
of an inequality of the form%
\[
\left\Vert Mf\right\Vert _{_{1}}\leq C\left\Vert
f\right\Vert _{p\left(  \cdot\right)  }%
\]
are considered in \cite{CUF2},\cite{FMS} and \cite{Has},  in case when $p(\cdot)$   satisfies a decay condition  and when $p(\cdot)$ is close to $1$ in value.


\medskip

To provide further information about the properties of $L^{\widetilde{q}%
(\cdot)}(I)$ we recall that given a Banach space $X,$ a sequence $\left\{
\left(  f_{n},g_{n}\right)  \right\}  _{n\in\mathbb{N}}$ $\subset X\times
X^{\ast}$ is said to be a biorthogonal system if, for all $m,n\in\mathbb{N},$%
\[
\left\langle g_{n},f_{n}\right\rangle =1\text{ and }\left\langle g_{m}%
,f_{n}\right\rangle =0\text{ for }m\neq n,
\]
where $\left\langle \cdot,\cdot\right\rangle $ denotes the duality pairing in
$X.$ Given a biorthogonal system $\left\{  \left(  f_{n},g_{n}\right)
\right\}  _{n\in\mathbb{N}},$ the sequence $\left\{  f_{n}\right\}
_{n\in\mathbb{N}}$ is called fundamental in $X$ if the closure of 
$\operatorname{span}\left\{f_{n}:n\in\mathbb{N} \right\}  $ is $X;$ it is a basis of $X$ if for
every $f\in X,$%
\[
f=\sum\nolimits_{n=1}^{\infty}\left\langle g_{n},f\right\rangle f_{n},
\]
with convergence in the norm of $X.$ If $\left\{  f_{n}\right\}
_{n\in\mathbb{N}}$ is a basis of $X,$ then $\left\{  g_{n}\right\}
_{n\in\mathbb{N}}$ is a basis in  the closed  linear span  $\overline{\operatorname{span}\left \{g_n:\, n\in \mathbb{N} \right\}}$.

\begin{theorem}
\label{Theorem 4.3}Let $\left\{  \left(  f_{n},g_{n}\right)  \right\}
_{n\in\mathbb{N}}$ be a biorthogonal system in \newline$L^{\widetilde{q}%
(\cdot)}(I)\times\left(  L^{\widetilde{q}(\cdot)}(I)\right)  ^{\ast}$ and
suppose $\left\{  g_{n}\right\}  _{n\in\mathbb{N}}$ is fundamental in $C(I).$
If $\left\{  f_{n}\right\}  _{n\in\mathbb{N}}$ is a basis in $L^{\widetilde
{q}(\cdot)}(I),$ then $\left\{  g_{n}\right\}  _{n\in\mathbb{N}}$ is a basis
in $C(I).$
\end{theorem}

\begin{proof}
This follows immediately from Theorem \ref{theorem 3.1}.

\medskip
\end{proof}

Further results of this kind are given in \cite{Kop1} and \cite{Kop2}.

\section{\bigskip Almost everywhere divergence of Fourier series in
$L^{p(\cdot)}\left(  \mathbb{T}\right)  $}

We conclude the paper by exhibiting the role played by the spaces of variable
exponent that we have been considering in connection with functions with
almost everywhere divergent Fourier series. To fix the notation, we denote as usual
$\mathbb{R}/(2\pi\mathbb{Z})$ by $\mathbb{T},$ and associate with any function
$f\in L^{1}\left(  \mathbb{T}\right)  $ its Fourier series%
\[
f(x)\sim\sum\nolimits_{-\infty}^{\infty}\widehat{f}(k)e^{ikx},
\]
where%
\[
\widehat{f}(k)=\frac{1}{2\pi}\int\nolimits_{\mathbb{T}}f(x)\exp(-ikx)dx.
\]
The $n^{th}$ partial sum of the trigonometric Fourier series of $f$ is%
\[
S_{n}(x,f):=\sum\nolimits_{k=-n}^{n}\widehat{f}(k)e^{ikx}.
\]
In \cite{Kol}, Kolmogorov constructed his famous example of a function $f\in
L^{1}\left(  \mathbb{T}\right)  $ such that its partial sums $S_{n}(x,f)$
diverge unboundedly almost everywhere. Later, Marcinkiewicz \cite{Mar}
produced a function in which the Fourier series diverged a.e. even though the
partial sums were bounded. Kolmogorov's function belongs to $L\log\log L;$
Chen \cite{Che} gave examples of functions in $L(\log\log L)^{1-\varepsilon}, $
$(0<\varepsilon<1)$,  that have a.e. divergent Fourier series; and Konyagin \cite{Kon} produced
functions, with similar bad properties, in the space $L\phi(L),$ where
$\phi(t)=o\left(  \sqrt{\log t/\log\log t}\right)  .$

The function spaces between $L^{1}\left(  \mathbb{T}\right)  $ and $\cup
_{p>1}L^{p}\left(  \mathbb{T}\right)  $  $(p={\rm constant})$ play an important role in the problem
of the a.e. convergence of Fourier series, since every $f\in\cup_{p>1}%
L^{p}\left(  \mathbb{T}\right)  $ has an a.e. convergent Fourier series, while
as shown by Kolmogorov there is a function $f\in L^{1}\left(  \mathbb{T}%
\right)  $ with a.e. divergent Fourier series. Further discussion of this
point is given in \cite{Bar}. Turning now to spaces with variable exponent, we
remark that in contrast to the situation for classical Lebesgue spaces,%
\begin{equation}
L^{1}\left(  \mathbb{T}\right)  =\cup L^{p\left(  \cdot\right)  }\left(
\mathbb{T}\right)  , \label{Eq 5.1}%
\end{equation}
where the union is over all $p(\cdot)\in\mathcal{P}\left(  \mathbb{T}\right)
$ (defined just as $\mathcal{P}\left(  I\right)  $ was defined in section 2)
that are greater than $1$ a.e. To establish this claim, let $f\in L^{1}\left(
\mathbb{T}\right)  $ and for each $n\in\mathbb{N}$ define%
\[
E_{n}=\left\{  x\in\mathbb{T}:n-1\leq\left\vert f(x)\right\vert <n\right\}  ;
\]
plainly $\sum\nolimits_{n}n\left\vert E_{n}\right\vert <\infty.$ Let $\left\{
\varepsilon_{n}\right\}  $ be a sequence of positive numbers such that
\[
\sum\nolimits_{n}n^{1+\varepsilon_{n}}\left\vert E_{n}\right\vert <\infty;
\]
for example, we could take $\varepsilon_{n}=1/n.$ Now define
$p(t)=1+\varepsilon_{n}$ $\left(  t\in E_n\right)  .$ It is apparent
that $f\in L^{p(\cdot)}\left(  \mathbb{T}\right)  .$ In view of (\ref{Eq 5.1})
it is natural to seek to characterise those spaces $L^{p(\cdot)}\left(
\mathbb{T}\right)  $ that contain functions with a.e. divergent Fourier
series. To prepare for a discussion of this question we give some details of
the procedure used in the construction of the Kolmogorov and the Marcinkiewicz
functions. The following lemma (see \cite{Bar}) is crucial for the
construction of both examples: condition (iii) of the lemma is necessary only
for the Marcinkiewicz example; for the Kolmogorov example it is sufficient for
conditions (i), (ii) and (iv) to be satisfied.

\begin{lemma}
\label{Lemma 5.1}There is a sequence of functions $\phi_{n}$ satisfying the
following conditions:

(i) For all $n\in\mathbb{N},$%
\[
\phi_{n}\geq0\text{ and }\int\nolimits_{0}^{2\pi}\phi_{n}(x)dx=2;
\]

(ii) each $\phi_{n}$ has bounded variation;

(iii) there is a sequence of subsets $H_{n}$ of $\left[  0,2\pi\right]  ,$
with%
\[
\lim_{n\rightarrow\infty}\left\vert H_{n}\right\vert =2\pi,
\]
such that there exists $A$ with the property that for all $n,r\in\mathbb{N}$
and all $x\in H_{n},$%
\[
\left\vert S_{r}\left(  x,\phi_{n}\right)  \right\vert \leq A\log n;
\]

(iv) if $\varepsilon>0,$ there exist $\alpha>0$ and $N\in\mathbb{N}$ such that
given any $n>N$ there is a set $E_{n}\subset\left[  0,2\pi\right]  $ for which

(a) $\left\vert E_{n}\right\vert >2\pi-\varepsilon,$

(b) for any $x\in E_{n},$ there exists $r_{x}$ $\in\mathbb{N}$ such that
$\left\vert S_{r_{x}}\left( x,\phi_{n})\right)  \right\vert >\alpha\log n,$

(c) $n\leq r_{x}\leq m_{n},$ where $m_{n}$ depends only on $n$ but not on
$\varepsilon.$
\end{lemma}

The proof is based on the following constructions of the functions $\phi_{n}.$
Let
\[
A_{k}=\frac{4\pi k}{2n+1}\text{ }\left(  k=1, 2, \cdots, n\right)
\]
and suppose that%
\[
\lambda_{1}=1, \lambda_{2}, \cdots, \lambda_{n},
\]
is an increasing sequence of odd numbers, chosen as detailed below. Let
\[
m_{1}=n, 2m_{k}+1=\lambda_{k}(2n+1)\text{ if }k=2, \cdots, n,
\]
and define non-overlapping intervals
\[
\Delta_{k}=\left(  A_{k}-\frac{1}{m_{k}^{2}},A_{k}+\frac{1}{m_{k}^{2}}\right)
\text{ \ }\left(  k=1, 2, \cdots, n\right)  .
\]
Let%
\[
\phi_{n}(x)=\left\{
\begin{array}
[l]{ll}%
m_{k}^{2}/n, &x\in\Delta_{k}\text{ }\left(  k=1,2, \cdots, n\right)  ,\\
0, &x\in[0,2\pi]\setminus \cup_{k=1}^{n}\Delta_{k},\\
\phi_{n}(x+2\pi)=\phi_{n}(x), &x\in\mathbb{R}.
\end{array}
\right.
\]
For $n\geq2$ let%
\[
D_{k}=\left[  A_{k}+\frac{1}{n\log n},A_{k+1}-\frac{1}{n\log n}\right]  \text{
\ }\left(  k=1,2,\cdots, n-1\right)
\]
and set%
\[
H_{n}=\cup_{k=1}^{n-1}D_{k}.
\]
The $m_{k}$ are defined inductively as follows: suppose we have $\lambda
_{1}< \cdots <\lambda_{k-1}$ $(k\geq2)$ and correspondingly $m_{1}< \cdots <m_{k-1}.$ We
may then choose $m_{k}$ so large that%
\begin{equation}
\left\vert \frac{1}{\pi}\int\nolimits_{\cup_{j=1}^{k-1}\Delta_{j}}\phi
_{n}(t)L_{m_{k}}(t-x)dt\right\vert <1\text{ \ }\label{Eq 5.2}%
\end{equation}
for all $x\in D_{k-1},$ where $L_{m_{k}}$ is the Dirichlet kernel of order
$m_{k}:$%
\[
L_{m_{k}}(t)=\frac{\sin\left(  m_{k}+\frac{1}{2}\right)  t}{2\sin(t/2)}.
\]
The choice of $m_{k}$ may be made as follows. For $x\in D_{k-1}$ and
$t\in\Delta_{j},$ )$(j=1, 2, \cdots, k-1)$,%
\[
\left\vert t-x\right\vert >\frac{1}{n\log n}-\frac{1}{n^{2}}>\frac{1}{2n\log
n};
\]
and the function%
\[
t\longmapsto\frac{\phi_{n}(t)}{2\sin\left\{  \left(  t-x\right)  /2\right\}  }%
\]
is bounded on every $\Delta_{j}$ $\left(  j=1, \cdots, k-1\right).$ Thus the
integral%
\[
\int\nolimits_{\Delta_{j}}\frac{\phi_{n}(t)}{2\sin\left\{  \left(  t-x\right)
/2\right\}  }\sin\left(  m_{k}+\frac{1}{2}\right)  (t-x)dt
\]
can be made as small as desired if $m_{k}$ is sufficiently large.

The examples of Kolmogorov and Marcinkiewicz are of the following form:
\[
K(x)=\sum\nolimits_{k=1}^{\infty}\frac{\phi_{n_{k}}(x)}{\sqrt{\log n_{k}}}%
\]
and%
\[
M(x)=\sum\nolimits_{k=1}^{\infty}\frac{\phi_{n_{k}}(x)}{\log n_{k}},
\]
respectively, where the sequence of integers $n_{k}$ is strictly increasing
and satisifes a number of conditions (see \cite{Bar}, pp. 437-439). For $K$ we
have%
\[
\overline{\lim}_{n\rightarrow\infty}\left\vert S_{n}(x,K)\right\vert
=\infty\text{ a.e.,}%
\]
while
\[
\overline{\lim}_{n\rightarrow\infty}\left\vert S_{n}(x,M)\right\vert
<\infty\text{ a.e.}%
\]
Note that in the construction involved in the proof of Lemma \ref{Lemma 5.1}
the symmetric intervals $\left(  A_{k}-\frac{1}{m_{k}^{2}},A_{k}+\frac
{1}{m_{k}^{2}}\right)  $ may be replaced by the non-symmetric intervals
$\left(  A_{k},A_{k}+\frac{2}{m_{k}^{2}}\right)  .$

In view of (\ref{Eq 5.1}) it is clear that the Kolmogorov example belongs to
some Lebesgue space with variable exponent $p$ with $p(x)>1$ a.e.; the same
holds for the Marcinkiewicz example. More information is provided by the
following theorem.

\begin{theorem}
\label{Theorem 5.2}There exists $p\in\mathcal{P}(\mathbb{T}),$ with
$1<p(x)<\infty$ a.e., such that the space of continuous functions
$C(\mathbb{T})$ is a closed subspace of $L^{p(\cdot)}$ $(\mathbb{T})$ and both
the Kolmogorov and the Marcinkiewicz example belong to $L^{q(\cdot)}$
$(\mathbb{T}),$ where $q$ is the conjugate of $p.$
\end{theorem}

\begin{proof}
Let
\[
t_{n}^{k}=\frac{4\pi k}{2n+1}\text{ \ }\left(  n\in\mathbb{N}%
,k=1,...,n\right)  ;
\]
the set $\left\{  t_{n}^{k}:n\in\mathbb{N},k=1,...,n\right\}  $ is a dense
subset of $\left[  0,2\pi\right]  .$ Noting that%
\[
\sum\nolimits_{n=1}^{\infty}\int\nolimits_{0}^{2/n^{2}}\log(1/x)dx<\infty,
\]
define $\delta_{n}^{k}$ $\left(  n\in\mathbb{N},k=1, \cdots, n\right)  $ in such a
way that $\delta_{n}^{1}=2/n^{2}$ and
\[
\sum\nolimits_{n=1}^{\infty}\sum\nolimits_{k=1}^{n}\int\nolimits_{t_{n}^{k}%
}^{t_{n}^{k}+\delta_{n}^{k}}m\left(  x-t_{n}^{k}\right)  dx<\infty.
\]
Now define the exponent $p$ by%
\[
p(x)=2+\sum\nolimits_{n=1}^{\infty}\sum\nolimits_{k=1}^{n}m\left(  x-t_{n}%
^{k}\right)  \chi_{\left(  t_{n}^{k},t_{n}^{k}+\delta_{n}^{k}\right)  }(x).
\]
Given any fixed $n\in\mathbb{N},$ choose the numbers $m_{k}$ in the estimate
(\ref{Eq 5.2}) so that
\[
t_{n}^{k}+\frac{2}{m_{k}^{2}}<\delta_{n}^{k}\text{ }\left(  k=2, \cdots, n\right)
.
\]
Finally, observe that by Lemma \ref{Lemma 4.1} both $K$ and $M$ belong to
$L^{q(\cdot)}$ $(\mathbb{T}),$ where $q$ is the conjugate of $p.$
\end{proof}

{\bf Acknowledgement.} We thank the anonymous referee for his/her remarks, which have improved the
final version of this paper.

\end{document}